\documentclass[a4paper]{amsart}
\usepackage [english]{babel}
 \usepackage{amsmath}
 \usepackage{epsfig}
 \usepackage{graphics}
 \usepackage{amscd}
\usepackage{enumitem}
 \usepackage{color}

\newtheorem{theorem}{Theorem}

\newtheorem{lemma}{Lemma}
\newtheorem{proposition}{Proposition}

\theoremstyle{definition}
\newtheorem{definition}{Definition}
\newtheorem{remark}{Remark}
\newtheorem{example}[theorem]{Example}

\newcommand{\R}{\mathbb{R}}
\newcommand{\N}{\mathbb{N}}

\newcommand{\V}{\mathcal{V}}

\newcommand{\D}{\partial}
\newcommand{\cl}[1]{\overline{#1}}
\newcommand{\dif}[1]{\,\mathrm{d}#1}
\newcommand{\wt}[1]{\widetilde{#1}}
\newcommand{\MxN}{M\times N}

\newcommand{\RxMNMN}{\R\times M\times N\times M\times N}
\newcommand{\CTMN}{C_T(M\times N)}
\newcommand{\X}{\times}

\newcommand{\giu}[1]{#1\sb\#} 
\newcommand{\su}[1]{#1\sp\#}

\newcommand{\h}[1]{\widehat{#1}}
\renewcommand{\t}[1]{\widetilde{#1}}

\def\overstrike#1#2{{\setbox0\hbox{$#2$}\hbox to \wd0{\hss
      $#1$\hss}\kern-\wd0\box0}}
\newcommand{\avrg}[1]{\overstrike{#1}{/}}

\DeclareMathOperator{\ind}{\mathrm{ind}}
\DeclareMathOperator{\sign}{\mathrm{sign}}

\title[Periodic perturbations with delay of coupled ODEs]%
{Periodic perturbations with delay of coupled differential 
equations on manifolds with application to a sunflower-like equation}
\author[L.\ Bisconti]{Luca Bisconti}
\email[L.\ Bisconti]{luca.bisconti@unifi.it}
\author[M.\ Spadini]{Marco Spadini}
\email[M.\ Spadini]{marco.spadini@unifi.it}

\begin{document}

\begin{abstract}
  We investigate the structure of the set of $T$-periodic solutions to
  periodically perturbed coupled delay differential equations on
  differentiable manifolds. By using fixed point index and
  degree-theoretic methods we prove the existence of branches of
  $T$-periodic solutions to the considered equations. As main
  application of our methods, we study a generalized version of the
  sunflower equation.
\end{abstract}

\maketitle

\section{Introduction}
In this paper we are concerned with a generalization of \cite{Sp06}
based on \cite{BiSp13}, which deals with the structure of the set of
harmonic solutions of periodically perturbed coupled ODEs on
manifolds. So doing, we bridge the gap between \cite{BeCaFuPe} and
\cite{BiSp13,FS09} in the sense that the main results of those papers
can be deduced from ours (see Remark \ref{genRit} below).  As a
further motivation we will present an application of our methods to a
generalized version of the well-known sunflower equation
(cf.~\eqref{eq:sunflower}), which is a second order delay differential
equation used to model the helical movement of the tip of a sunflower
plant (see, e.g., \cite{Somolinos, DelayDiffEq}).

Let us describe more precisely our setting. Let $M\subseteq\R^k$ and
$N\subseteq\R^s$ be boundaryless smooth manifolds, let
$\mathfrak{f}\colon\RxMNMN\to\R^k$ be tangent to $M$, and let
$\mathfrak{g}\colon\MxN\to\R^s$ and $\mathfrak{h}\colon\RxMNMN\to\R^s$
be tangent to $N$: This means that, for any $(t,p,q,v,w)\in\RxMNMN$,
then $\mathfrak{g}(p,q,v,w)$ and $\mathfrak{h}(t,p,q,v,w)$ belong to
the tangent space $T_{q}N$, and $\mathfrak{f}(t,p,q,v,w)$ is in
$T_pM$, respectively. Let also $a\colon\R\to\R$ be continuous. Given
$T>0$, we assume that $\mathfrak f$, $\mathfrak h$ and $a$ are
$T$-periodic in the $t$ variable. Consider the following system of
delay differential equations for $\lambda\geq 0$:
\begin{equation}\label{main}
  \left\{
    \begin{array}{l}
      \dot x(t)=\lambda \mathfrak{f}\big(t,x(t),y(t),x(t-r),y(t-r)\big),\\
      \dot y(t)=a(t)\mathfrak{g}\big(x(t),y(t)\big)+
      \lambda \mathfrak{h}\big(t,x(t),y(t),x(t-r),y(t-r)\big),
    \end{array}
  \right.
\end{equation}
where the time lag $r>0$ is given. This system is equivalent to a
single parameter-dependent delay differential equation on the product
manifold $\MxN\subseteq\R^{k+s}$.

Denote by $C_T(M)$ and $C_T(N)$ the spaces of $T$-periodic continuous
functions from $\R$ to $M$ and $N$, respectively, with the topology of
uniform convergence. We investigate the properties of the set of the
$T$-periodic triples (or briefly $T$-triples) of \eqref{main}, i.e.\
of those triples $(\mu, x, y)\in [0,\infty)\times C_T(M)\times
C_T(N)$, where $(x,y)$ is a solution to \eqref{main} when $\lambda
=\mu$.  In particular, we shall give conditions for the existence of a
noncompact connected component of \emph{nontrivial $T$-triples} (which
we call a ``\emph{branch}'') emanating from the set $\nu^{-1}(0)$,
where $\nu:\MxN\to\R^{k+s}$ is the vector field, tangent to
$\MxN\subseteq\R^{k+s}$, given by
\begin{equation}\label{eq:utility-map-for-the-degree}
  \nu(p,q)=\big(w_\mathfrak{f}(p,q),\,  \mathfrak{g}(p,q)\big),
\end{equation}
with
$w_\mathfrak{f}(p,q):=\frac{1}{T}\int_0^T\mathfrak{f}(t,p,q,p,q)\dif{t}$.
In the present setting, a $T$-triple $(\lambda, x, y)$ of \eqref{main}
is said to be \emph{trivial} if $\lambda = 0$ and $(x, y)$ is
constant.

It is tempting to try to achieve the desired generalization of \cite{Sp06} by simply 
using a time-transformation as in \cite{BIS:2011, Sp03} to get rid of the factor $a(t)$ 
in \eqref{main} and then adapt the argument of \cite{BeCaFuPe, FS09} to the present
case.  Nevertheless, this simple procedure does not work because the transformed 
perturbing term would result in a form inappropriate for our methods.  In fact, the 
time-transformation used in \cite{Sp03} does not preserve the fixed-delay structure. 
Instead, to prove our result, we follow the lead of \cite{BiSp13} and combine the techniques 
of \cite{Sp03} and \cite{Sp06}. 

\smallskip
As main application of our methods we study the set of periodic
solutions of a parametrized second order delayed differential equation
(the one called \emph{sunflower-like equation}).
This parametrized equation is derived in Section~\ref{sec:sunflower-like} 
in a fairly direct way starting from the sunflower equation, and reads as follows:
\begin{equation*} 
    \ddot y(t) = a(t) \dot y(t) + \lambda \phi\big(y(t), y(t-r)\big) ,
    \,\, \lambda \geq 0,
\end{equation*}
where $a\colon\R\to\R$ and $\phi\colon\R^2\to\R$ are continuous and $a$ is $T$-periodic 
with $\avrg a\ne 0$.  By using 
elementary transformations we will show how the above equation can be 
equivalently rewritten as systems of delay differential equations of type \eqref{main} 
to which our methods apply.  Thus, we find in a natural way a result about the structure
of the set of $T$-periodic solutions of the sunflower-like equation above.
  
\section{Poincar\'e--type translation operator}\label{secPoiType}

Consider the system of equations \eqref{main}. We are interested in its $T$-periodic 
solutions.  Without loss of generality, as suggested in \cite{Fr07}, we will assume
that $T\geq r$. In fact, for $n\in\N$, the system
\eqref{main} and
\begin{equation*}
  \left\{
    \begin{array}{l}
      \dot x(t)=\lambda \mathfrak{f}\Big(t,x(t),y(t),x\big(t-(r-nT)\big),y\big(t-(r-nT)\big)\Big)\\
      \dot y(t)=a(t)\mathfrak{g}\big(x(t),y(t)\big)
      +\lambda \mathfrak{h}\Big(t,x(t),y(t),x\big(t-(r-nT)\big),y\big(t-(r-nT)\big)\Big),
    \end{array}\right.
\end{equation*}
have the same $T$-periodic solutions. Thus, if necessary, one can
replace $r$ with $r-nT$, where $n\in\N$ is such that $0<r-nT\leq T$.

\smallskip 
Let us now introduce some notation. Given any $X\subseteq\R^k$, $\t X$ denotes the metric space 
$C\big([-r,0],X)$ with the distance inherited from the Banach space $\t\R^k=C([-r,0],\R^k)$ with 
the usual supremum norm. 

Given any $(p,q)\in\MxN$, denote by $\su p\in\t M$ and $\su q\in\t N$ the constant functions 
$\su p(t)\equiv p$ and $\su q(t)\equiv q$, $t\in[-r,0]$, respectively. Thus, 
$(\su p,\su q)\in\t\MxN\simeq\t M\X\t N$.  For any $U\subseteq\MxN$, define 
$\su U=\big\{(\su p,\su q)\in\t\MxN:(p,q)\in U\big\}$. Also, given $W\subseteq\t\MxN$, we put 
$\giu W=\big\{(p,q)\in\MxN:(\su p,\su q)\in W\big\}$.  Finally, we will
denote by $C_T(X)$ the metric subspace of the Banach space
$\big(C_T(\R^k),\,\|\cdot\|\big)$ of all the $T$-periodic continuous
maps $x:\R\to X$ (as above, with the usual $C^0$ norm). Observe that
$C_T(X)$ is complete if and only if $X$ is complete (or, equivalently, closed as a subset of
$\R^k$). Nevertheless, since $M$ and $N$ are locally compact,
$C_T(\MxN)\simeq C_T(M)\X C_T(N)$ is always locally complete.

\smallskip Assume now, unless differently stated, that $a$,
$\mathfrak{f}$, $\mathfrak g$ and $\mathfrak h$ are $C^1$. Consider
the map $H$ with domain $\mathcal{D}_H\subseteq\R\X\t M\X\t N\X\R$ in
$\t M\X\t N$ defined by
\begin{equation*}
  H(\lambda,\varphi,\psi, \mu)(\theta)
  =\Big(x_{\lambda,\mu}\big(\varphi,\psi,T+\theta\big),
  y_{\lambda,\mu}\big(\varphi,\psi,T+\theta\big)\Big),\quad
  \theta\in[-r,0],
\end{equation*}
where
$t\mapsto\big(x_{\lambda,\mu}(\varphi,\psi,t),y_{\lambda,\mu}(\varphi,\psi,t)\big)$
denotes the unique maximal solution of the initial-value problem
\[
\left\{
    \def\arraystretch{1.4}
  \begin{array}{lr}
    \!{\begin{aligned}
    &\!\!\dot x(t)=\lambda \big[\mu \mathfrak{f}\big(t,x(t),y(t),x(t-r),y(t-r)\big)
               +(1-\mu)\tfrac{a(t)}{\avrg{a}}w_\mathfrak{f}\big(x(t),y(t)\big)\big],\\
    &\!\!\dot y(t)=a(t)\mathfrak{g}\big(x(t),y(t)\big)
                +\lambda\mu \mathfrak{h}\big(t,x(t),y(t),x(t-r),y(t-r)\big),
    \end{aligned}} & t>0,\\
    \!\!x(t)=\varphi(t),\; y(t)=\psi(t), & \makebox[0pt][r]{$t\in[-r,0]$.}
  \end{array}
\right.
\]
Well known properties of differential equations imply that $\mathcal{D}_H$ is an open 
subset of $\R\X\t M\X\t N\X \R$.  A similar argument shows that the set
$\mathcal{D}':=\{(\varphi,\psi)\in\wt M\X\wt N : (0,\varphi,\psi,1)\in\mathcal{D}_H\}$
is open as well. Also, since we are assuming $T\geq r$ (see above), the Theorem of 
Ascoli-Arzel\`a implies that $H$ is a locally compact map (compare, e.g.\ \cite{Ol69} or 
\cite{BeCaFuPe1}).

\begin{remark}\label{rem:CPP}
Consider the following equation:
  \begin{equation}\label{unperturbed}
    \left\{
      \begin{array}{l}
        \dot x(t)=0,\\
        \dot y(t)=a(t)\mathfrak{g}\big(x(t),y(t)\big),
      \end{array}\right.
  \end{equation}
Given $V\subseteq\wt M\X\wt N$ such that $\cl V\subseteq\mathcal{D}'$ we have that all 
solutions of \eqref{unperturbed} starting at time $t=0$ from $\cl{\giu V}$ are defined 
(at least) for $t\in[0,T]$. An argument similar to, e.g., \cite[Remark 2.3]{Sp03} or 
\cite[Remark 2.1]{BiSp13}) shows that the same assertion holds for \eqref{unperturbed} 
when $a(t)$ is replaced with its average $\avrg{a}$:
  \begin{equation}\label{unp0}
    \left\{
      \begin{array}{l}
        \dot x(t)=0,\\
        \dot y(t)=\avrg{a}\mathfrak{g}\big(x(t),y(t)\big),
      \end{array}\right.
  \end{equation} 
In fact, one could prove that solutions of \eqref{unperturbed}  and of \eqref{unp0},
leaving at time $t=0$ from the same point, coincide at time $t=T$. Thus, $T$-periodic
orbits (images of solutions) of \eqref{unperturbed} and \eqref{unp0} must coincide.
More precisely,  let $\{\Phi_t\}_{t\in\R}$ be the local flow associated to \eqref{unp0}. 
That is, $\Phi\colon U\to\MxN$ is defined on an open subset $U$ of $\R\X\MxN$, containing 
$\{0\}\times\MxN$, with the property that for any $(p,q)\in\MxN$ the curve 
$t\mapsto\Phi_t(p,q)$ is the maximal solution of \eqref{unp0} given the initial condition 
$\Phi_0(p,q)=(p,q)$. 
Then, given $\tau\in\mathbb{R}$, the domain of $\Phi_\tau$ is the open set consisting of 
those points $(p,q)\in\MxN$ for which the maximal solution of \eqref{unp0} starting from 
$(p,q)$ at $t=0$ is defined up to $\tau$. (We are interested, in particular, to the case 
$\tau=T$.) Let $\{\Psi_t\}_{t\in\R}$ be the anologous local flow associated to 
\eqref{unperturbed}. The argument of the above cited remarks show that 
$\Psi_T(p,q) = \Phi_T(p,q)$ whenever this relation makes sense, in particular for all 
$(p,q)\in\cl{\giu V}$.
\end{remark}

\smallskip The following definition is convenient:
\begin{definition}\label{defCPP}
  We say that $V\subseteq \wt M\X\wt N$ has the \emph{constant periodic property for 
  \eqref{unperturbed}} if any $T$-periodic solution $(x,y)$ of Equation \eqref{unperturbed}
  that intersects $\D \giu V$ is constant.  
\end{definition}

We have the following result:

\begin{lemma}\label{lemma1}
  Let $V\subseteq \wt M\X\wt N$ be open and such that
  \[
  Z_V:=\big\{ (\su p,\su q)\in V:\nu(p,q)=0\big\}
  \]
  is compact. Then, there exists an open neighborhood $W\subseteq V$
  of $Z_V$ and $\varepsilon>0$ s.t.\ $[0,\varepsilon]\X \cl{W}\X
  [0,1]\subseteq\mathcal{D}_H$ and $H\big([0,\varepsilon]\X \cl{W}\X
  [0,1]\big)$ is compact.

  Assume in addition that $\giu V$ is relatively compact, $\cl V\subseteq\mathcal{D}'$ and 
  that $V$ has the constant periodic property for \eqref{unperturbed} (Definition \ref{defCPP}). 
  Then $W$ can be taken in such a way that it has the constant periodic property as well. That 
  is, if $(x,y)$ is a $T$-periodic solution of \eqref{unperturbed} intersecting $\D\giu W$, then 
  $(x,y)$ is constant.
\end{lemma}
\begin{proof}
  One immediately checks that the set $Z_V$ consists of $T$-periodic solutions of 
  \eqref{unperturbed}.  Thus, we have that $Z_V\subseteq\mathcal{D}'$ and the first part of the
  lemma follows from the local compactness of $H$.
 
   Let us now prove the second part of the assertion. Let $\{\Phi_t\}_{t\in\R}$ be the local 
   flow associated to \eqref{unp0} as in Remark \ref{rem:CPP}. 
  The map $(t,p,q)\mapsto\Phi_t(p,q)$ is continuous and, therefore, the ``attainable set''
  $\mathcal{A}_T:=\Phi_{[0,T]}\big(\cl{\giu{V}}\big)$ is  compact. Thus, the union 
  $\mathcal{O}_T$ of all $T$-periodic orbits of \eqref{unp0} starting from points 
  of $\cl{\giu{V}}$, being closed in $\mathcal{A}_T$, is compact as well. Clearly, 
  since \eqref{unp0} is autonomous, $\mathcal{O}_T$ is actually the set of all
  $T$-periodic orbits of \eqref{unp0} that intersect $\cl{\giu{V}}$.

  Remark \ref{rem:CPP} shows that $\mathcal{O}_T$ consists indeed of all $T$-periodic orbits 
  of \eqref{unperturbed} that intersect $\cl{\giu{V}}$. 
  Let us denote by $K$ the union of $Z_V$ with this set. Clearly $K$
  is contained in $\mathcal{D}'$. The local compactness of $H$ implies
  the existence of an open neighborhood $W\subseteq V$ of $K$ and a
  positive $\varepsilon$ with the property that
  $[0,\varepsilon]\X \cl{W}\X [0,1]\subseteq\mathcal{D}_H$ and
  $H\big([0,\varepsilon]\X \cl{W}\X [0,1]\big)$ is compact.
  The second part of the claim follows now from the fact a $T$-periodic
  solution of \eqref{unperturbed} whose image intersects the boundary $\D \giu{W}$,
  of the set $W$ just constructed, necessarily intersects $\D \giu{V}$ and thus 
  must be constant.
\end{proof}

It is convenient to set
\[
Q_T^\lambda=H(\lambda,\cdot,\cdot,1),
\quad\text{and}\quad 
\wt Q_T^\lambda = H(\lambda,\cdot,\cdot,0).
\]
We will denote the domain of $H(\cdot,\cdot,\cdot,1)$ by the letter $\mathcal{D}$.

The following is the main result of this section (cf.\ \cite{FS09, BiSp13}). It relates 
the fixed point index of $Q_T^\lambda$ for small $\lambda>0$ (see, e.g., \cite{Mi, N} for 
an introduction) with the degree of the tangent vector field $\nu$. Recall that this 
notion, roughly speaking, counts (algebraically) the zeros of a vector field; for an 
exposition of this topic we refer, e.g., to \cite{Mi} or \cite{FuPeSp}.

\begin{theorem}\label{fixind}
  Given $V\subseteq\wt M\X\wt N$ open and such that
  \begin{enumerate}[label=(\roman*)]
  \item $\giu V$ is relatively compact;
  \item There exists $s>0$ such that $[0,s]\X\cl V\subseteq\mathcal{D}$;
  \item $Z_V$ is compact;
  \item \label{itm:4} If $(x,y)$ is a $T$-periodic solution of
    \eqref{unperturbed} whose image intersects $\D\giu V$, then $(x,y)$ is constant.
  \end{enumerate}
  Then there exists $\lambda_*\in(0,s]$ such that, for $\lambda\in(0,\lambda_*)$, 
  $\ind(Q^\lambda_T,V)$ is well defined and
  \[
  \ind(Q^\lambda_T,V)=\sign(\avrg{a})^{\dim N}\deg\big(-\nu,
  \giu{V}\big).
  \]
\end{theorem}
\smallskip

The symbol ``$\ind(Q^\lambda_T,V)$'' in the above formula denotes the fixed point index of
$Q^\lambda_T$ in the open set $V$, whereas ``$\deg\big(-\nu,\giu V\big)$'' denotes the
degree of the tangent vector field $-\nu$ in the open subset $\giu V$ of $M\X N$.

\begin{proof}[Proof of Theorem~\ref{fixind}]
  Let $W$ and $\varepsilon$ be as in Lemma~\ref{lemma1}. Consider the sets
  \begin{gather*}
    \mathcal{S}=\big\{(\lambda,\varphi,\psi)\in[0,\varepsilon]\X\cl W :
    H(\lambda,\varphi,\psi,1)=(\varphi,\psi)\big\},\\
    \mathcal{S}_0=\mathcal{S}\cap \big(\{0\}\X\wt M\X\wt N\big).
  \end{gather*}
  Clearly, $\mathcal{S}$ is compact being a closed subset of the compact set 
  $[0,\varepsilon]\X H\big([0,\varepsilon]\X\cl{W}\X [0,1]\big)$. Thus $\mathcal{S}_0$ is 
  compact as well. Using the definition of $Q^\lambda_T$, we will prove the following fact:

  \smallskip
  \noindent\textbf{Claim~1.} There exists $\lambda_0\in\left(0,\min\{\varepsilon,s\}\right]$ 
  such that if $(\varphi,\psi)\in V$ is a fixed point of $Q^\lambda_T$ with 
  $\lambda\in (0,\lambda_0]$ then $(\varphi,\psi)\in W$. That is, $Q^\lambda_T$ has no fixed 
  points in $\cl V\setminus W$ for $\lambda\in (0,\lambda_0]$.

  To prove this claim we proceed by contradiction. If the claim is false there exist sequences 
  $\{\lambda_n\}\subseteq (0,\lambda_0]$, and 
  $\big\{(\varphi_{n},\psi_{n})\big\}\subseteq\cl V\setminus W$, with $\lambda_n\to 0$ and 
  $(\lambda_n,\varphi_{n},\psi_{n})\in\mathcal{S}$.  
  By the compactness of $\mathcal{S}_0\cap(\cl V\setminus W)$ we can assume that
  $(\varphi_{n},\psi_{n})\to(\varphi_0,\psi_0)\in \mathcal{S}_0\cap(\cl V\setminus W)$. The
  continuous dependence on data shows that the solution of \eqref{unperturbed} with initial 
  data $(\varphi_0,\psi_0)$ is $T$-periodic. Assumption \ref{itm:4} shows that there exists  
  $p_0\in M$ and $q_0\in N$ such that $(\varphi_0,\psi_0)=(\su p_0,\su q_0)$. Clearly, one
  has $\mathfrak{g}(p_0,q_0)=0$.  Let $(x_n,y_n)$ be the unique maximal
  solution of
  \begin{equation*}
    \left\{
    \def\arraystretch{1.3}
      \begin{array}{lr}
        \!\begin{aligned}
          & \dot x(t)=\lambda_n \mathfrak{f}\big(t,x(t),y(t),x(t-r),y(t-r)\big),\\
          & \dot y(t)=a(t)\mathfrak{g}\big(x(t),y(t)\big)+
             \lambda_n \mathfrak{h}\big(t,x(t),y(t),x(t-r),y(t-r)\big),
         \end{aligned} & t>0,\\    
         x(t)=\varphi_{n}(t),\quad y(t)=\psi_{n}(t),  & t\in[-r,0].
      \end{array}
    \right.
  \end{equation*}
  Then,
  \begin{equation*}
    0=x_n(T)-x_n(0)=\lambda_n \int_0^T
    \mathfrak{f}\big(t,x_n(t),y_n(t),x_n(t-r),y_n(t-r)\big)\dif{t}.
  \end{equation*}
  So that, in particular,
  \[
  0=\int_0^T
  \mathfrak{f}\big(t,x_n(t),y_n(t),x_n(t-r),y_n(t-r)\big)\dif{t}
  \]
  and, passing to the limit, we get
  \[
  0=\int_0^T
  \mathfrak{f}(t,p_0,q_0,p_0,q_0)\dif{t}=w_\mathfrak{f}(p_0,q_0).
  \]
  Hence, $\nu(p_0,q_0)=\big(w_\mathfrak{f}(p_0,q_0), \mathfrak{g}(p_0,q_0)\big)=0$. This 
  contradicts the choice of $W$ and completes the proof of Claim~1.  \smallskip

  Claim 1 shows that, for $\lambda\in (0,\lambda_0]$, the set of the fixed points of 
  $Q^\lambda_T$ that lie in $V$ is, in fact, contained in $W$. Hence, by the compactness 
  of $\mathcal{S}$, it is compact. As a consequence, $\ind(Q^\lambda_T,V)$ and
  $\ind(Q^\lambda_T,W)$ are well-defined and, by the excision
  property,
  \begin{equation}\label{iddegexc}
    \ind(Q^\lambda_T,V)=\ind(Q^\lambda_T,W),\quad\text{for $\lambda\in (0,\lambda_0]$}.
  \end{equation}

  In fact, when $\lambda$ is sufficiently small, something more can be
  obtained:

  \smallskip
  \noindent\textbf{Claim~2.} There exists $\lambda_*\in(0,\lambda_0]$,
  such that the homotopy $H_{\lambda}:\cl W\X[0,1]\to\wt M\X\wt N$
  given by
  $H_{\lambda}(\varphi,\psi,\mu)=H(\lambda,\varphi,\psi,\mu)$, is
  admissible for each $\lambda\in(0,\lambda_*]$.

  To prove the claim we ought to show that for each $\lambda\in(0,\lambda_*]$, $\lambda_*>0$ 
  sufficiently small, the set of fixed points
  \[
  \mathcal{F}_{\lambda}=\big\{(\varphi,\psi)\in\cl W:
  H(\lambda,\varphi,\psi,\mu)=(\varphi,\psi),\, \text{for some
    $\mu\in[0,1]$}\big\},
  \]
  which is compact being a closed subset of $H\big([0,\varepsilon]\X\cl{W}\X [0,1]\big)$, 
  is contained in $W$. Suppose by contradiction that this is not the case, that is, that 
  such a choice of $\lambda_*$ cannot be done. Then there are sequences 
  $\{\lambda_n\}\subseteq (0,\lambda_0]$, $\{\mu_n\}\subseteq [0,1]$ and 
  $\big\{(\varphi_n,\psi_n)\big\}\subseteq\partial W$ with $\lambda_n\to 0$ and
  \begin{equation}\label{fixptH}
    H(\lambda_n,\varphi_n,\psi_n,\mu_n)=(\varphi_n,\psi_n).
  \end{equation}
  As in the proof of Claim 1, by the compactness of
  $H\big([0,\varepsilon]\X \cl{W}\X [0,1]\big)$ we can assume that
  $(\varphi_n,\psi_n)\to(\varphi_0,\psi_0)\in\partial W$. The continuous dependence on data
  shows that the solution of \eqref{unperturbed} with initial data $(\varphi_0,\psi_0)$ is 
  $T$-periodic. Assumption \ref{itm:4} shows that there exists $p_0\in M$ and $q_0\in N$ such 
  that $(\varphi_0,\psi_0)=(\su p_0,\su q_0)$.  Clearly, we get $\mathfrak{g}(p_0, q_0)=0$.  
  From \eqref{fixptH} it follows that if $(x_n,y_n)$ is the solution of
  \[
  \left\{
    \def\arraystretch{1.3}
    \begin{array}{lr}
      \!\begin{aligned}
         &\! \dot x(t)=\lambda_n \big[\mu_n \mathfrak{f}\big(t,x(t),y(t),x(t-r),y(t-r)\big)\\
              &\hspace{4cm} 
                         +(1-\mu_n)\tfrac{a(t)}{\avrg{a}}w_\mathfrak{f}\big(x(t),y(t)\big)\big],\\
         &\! \dot y(t)=a(t)\mathfrak{g}\big(x(t),y(t)\big)\\
             &\hspace{3cm}  +\lambda_n\mu_n \mathfrak{h}\big(t,x(t),y(t),x(t-r),y(t-r)\big),
       \end{aligned} & t>0,\\
      \! x(t)=\varphi_n(t),\quad y(t)=\psi_n(t), & t\in[-r,0].
    \end{array}
  \right.
  \]
  Then,
  \begin{multline*}
    0=x_n(T)-x_n(0)=\lambda_n\int_0^T \mu_n
    \mathfrak{f}\big(t,x_n(t),y_n(t),x_n(t-r),y_n(t-r)\big)\dif{t}\\
    +\lambda_n\int_0^T(1-\mu_n)
    \frac{a(t)}{\avrg{a}}w_\mathfrak{f}\big(x_n(t),y_n(t)\big)\dif{t}.
  \end{multline*}
  So that
  \begin{multline*}
    0= \mu_n\int_0^T \mathfrak{f}\big(t,x_n(t),y_n(t),x_n(t-r),y_n(t-r)\big)\dif{t}\\
    +(1-\mu_n)\int_0^T\frac{a(t)}{\avrg{a}}w_\mathfrak{f}\big(x_n(t),y_n(t)\big)\dif{t}.
  \end{multline*}
  Passing to the limit we get
  \begin{multline*}
    0 = \mu_0\int_0^T
    \mathfrak{f}\big(t,p_0,q_0,p_0,q_0\big)\dif{t}
    +(1-\mu_0)\int_0^T\frac{a(t)}{\avrg{a}}w_\mathfrak{f}
    \big(p_0,q_0\big)\dif{t}=\\
     =\mu_0 w_\mathfrak{f}\big(p_0,q_0\big) + (1 - \mu_0 )
      w_\mathfrak{f}\big(p_0,q_0\big) =w_\mathfrak{f}(p_0,q_0), 
  \end{multline*}
  which contradicts the choice of $W$ and proves Claim~2.

  \smallskip 
  Claim 2, along with the homotopy invariance property, imply that for $\lambda\in(0,\lambda_*]$
  \begin{equation}\label{iddeghom}
    \ind(Q^\lambda_T,W)=\ind\big(H(\lambda,\cdot,\cdot,1),W\big)
    =\ind\big(H(\lambda,\cdot,\cdot,0),W\big) =\ind(\wt Q_T^\lambda,W).
  \end{equation}
  
  Consider the tangent vector field $v_\lambda$ on $M\X N$ given by
  \begin{equation*}
    v_\lambda(p,q):=\left(\frac{\lambda}{\avrg{a}}w_\mathfrak{f}(p,q),
      \lambda \mathfrak{g}(p,q)\right).
  \end{equation*}

  Theorem 3.2 of \cite{BiSp13} imply that, for each fixed $\lambda\in(0,\lambda_*]$
  \begin{equation}\label{iddegBS}
    \ind(\wt Q_T^\lambda,W)=\sign(\avrg{a})^{\dim(\MxN)}\deg(-v_\lambda,\giu W).
  \end{equation}
  Since $\lambda>0$, a well known property of the degree yields
  \begin{equation}\label{iddev1}
    \deg(-v_\lambda, \giu W)=\deg(-v_1, \giu W).
  \end{equation} 
  Lemma~1 of \cite{Sp06} shows that
  \begin{equation} \label{degdeg-v} \deg(-v_1, \giu
    W)=\sign(\avrg{a})^{\dim M}\deg (-\nu,\giu W),
  \end{equation}
  hence, by equalities \eqref{iddeghom}---\eqref{degdeg-v}, 
  taking into account that $\dim(\MxN)=\dim M+\dim N$, we get
  \begin{equation}\label{iddeghBSS}
    \begin{split}
      \ind(Q^\lambda_T,W) &= \ind(\wt Q_T^\lambda,W) 
                           = \sign(\avrg{a})^{\dim(\MxN)}\deg(-v_\lambda,\giu W)\\
      &=\sign(\avrg{a})^{\dim(\MxN)}\deg(-v_1,\giu W)\\
      &=\sign(\avrg{a})^{\dim(\MxN)}\sign(\avrg{a})^{\dim(M)}\deg(-\nu,\giu W)\\
      &=\sign(\avrg{a})^{2\dim M+\dim N}\deg(-\nu,\giu W)\\
      &=\sign(\avrg{a})^{\dim N}\deg(-\nu,\giu W).
    \end{split}
  \end{equation}
  Finally, by \eqref{iddegexc}, \eqref{iddeghBSS} and the excision property of the degree, we 
  get
  \begin{align*}
    \ind(Q^\lambda_T,V) & =\ind(Q^\lambda_T,W) \\
    & =\sign(\avrg{a})^{\dim N}\deg(-\nu,\giu W)\\
    & =\sign(\avrg{a})^{\dim N}\deg(-\nu,\giu V),
  \end{align*}
  which proves the assertion.
\end{proof}

\section{Branches of $T$-periodic solutions} \label{sec:branches}

Let $T>0$ be given, by $C_T (\R^d)$ we mean the Banach space of all
the continuous $T$-periodic functions $\zeta\colon\R\to\R^d$ whereas
$C_T (X)$ denotes the metric subspace of $C_T (\R^d)$ consisting of
all those $\zeta\in C_T (\R^d)$ that take values in $X$. It is not
difficult to prove that $C_T (X)$ is complete if and only if $X$ is
closed in $\R^d$.

It is also convenient to introduce the following notation: Given $(p,q)$ 
in $\MxN$, let $(\h p,\h q)\in\CTMN=C_T(\R, M\X N)$ be the constant maps 
$\big(\h p(t),\h q(t)\big)\equiv (p,q)$, $t\in\R$.

%
%
%

We are now in a position to state our result concerning the
``branches'' of $T$-triples of \eqref{main}. Its proof follows closely
the one of Th.\ 5.1 in \cite{FS09} (see also \cite{FuPe6, BiSp13}), for
this reason we only provide a sketch for the sake of completeness.

\begin{theorem}\label{tunoRit}
  Let $\Omega$ be an open subset of $[0,\infty)\times \CTMN$, and let 
  $\Omega_{\MxN}:=\left\{(p,q)\in\MxN:(0,\h p,\h q)\in\Omega\right\}$.
  Assume that $\deg\big(\nu,\Omega_{\MxN}\big)$ is well-defined and 
  nonzero. Then there exists a connected set $\Gamma$ of nontrivial 
  $T$-triples for \eqref{main} in $\Omega$ whose closure in
  $[0,\infty)\times\CTMN$ meets $\nu^{-1}(0)\cap\Omega_{\MxN}$ and is 
  not contained in any compact subset of $\Omega$. In particular, if
  $\MxN$ is closed in $\R^{k+s}$ and $\Omega =[0,\infty )\times
  \CTMN$, then $\Gamma$ is unbounded.
\end{theorem}

The proof of this theorem is based on the following global connection
result (see \cite{FuPe6}), which will also be needed later.

\begin{lemma}\label{L.3.1}
  Let $Y$ be a locally compact metric space and let $Z$ be a compact
  subset of $Y$.  Assume that any compact subset of $Y$ containing $Z$
  has nonempty boundary. Then $Y\setminus Z$ contains a connected set
  whose closure (in $Y$) intersects $Z$ and is not compact.
\end{lemma}

We are now ready to sketch the proof of the theorem.

\begin{proof}[Sketch of the proof of Theorem \ref{tunoRit}]
  This proof can be roughly divided into three steps:\smallskip

\noindent\textbf{Step 1.} We assume first that the maps $a$, $\mathfrak{f}$, $\mathfrak{g}$
and $\mathfrak{h}$ are $C^1$, so that uniqueness of solutions holds
for \eqref{main}. Consider the following notion:

\noindent
A triple $(\lambda,\varphi,\psi)\in [0,\infty)\X\t M\X\t N$ is said to
be a \emph{starting triple} for \eqref{main} if the following initial
value problem has a $T$-periodic solution:
\begin{equation}\label{due.uno}
  \left\{
    \begin{array}{lr}
      \begin{aligned}
        &\dot x(t) = \lambda\mathfrak{f}\big(t,x(t),y(t),x(t-r),y(t-r)\big)\\
        &\dot y(t) = a(t)\mathfrak{g}\big(x(t),y(t)\big)
        +\lambda\mathfrak{h}\big(t,x(t),y(t),x(t-r),y(t-r)\big)
      \end{aligned} &  t>0,\\[5mm]
      \begin{aligned}
        &x(t)=\varphi(t),\\
        &x(t)=\psi(t)
      \end{aligned}
      & t\in [-r,0].
    \end{array}
  \right.
\end{equation}
A triple of the type $(0,\su p, \su q)$ with $g(p,q)=0$ is clearly a
starting triple and will be called a \emph{trivial starting
  triple}. The set of all starting triples for \eqref{main} will be
denoted by $S$. By known continuous dependence properties of delay
differential equations the set $\V\subseteq[0,\infty)\X\t M\X\t N$ of
all triples $(\lambda,\varphi,\psi)$ such that the unique solution of
\eqref{due.uno} is defined at least up to $T$ is open (compare it to
the set $\mathcal{D}$ defined in section \ref{secPoiType}). Clearly
$\V$ contains the set $S$ of all starting pairs for \eqref{main}.

Given an open set $W$ of $[0,\infty)\X\t M\X\t N$, let 
\[
\giu{W^0}:=\giu{\left(W\cap(\{0\}\X\t M\X\t N)\right)}
          =\left\{(p,q)\in M\X N: (0,\su p,\su q)\in W\right\}.
\]
Our first step consists of proving that, if $\deg\big(\nu,\giu{W^0}\big)$ 
is well-defined and nonzero, then there exists in $S\cap W$ a connected
set $\mathcal{G}$ of nontrivial starting triples whose closure in
$S\cap W$ meets $\big\{(0,\su p,\su q)\in W\cap\V:g(p,q)=0\big\}$ and is
not compact.

The proof of this fact follows closely the one of Proposition 4.1 in
\cite{FS09} using Theorem \ref{fixind} in place of \cite[Th.\
3.2]{FS09}. Loosely speaking, this proof uses the properties of the
fixed point index and of the degree of a tangent vector field to
obtain a contradiction with Lemma \ref{L.3.1}. (Compare also
\cite[Th.\ 4.1]{BiSp13}.)\smallskip

\noindent\textbf{Step 2.} As in Step 1 we assume that the maps $a$, $\mathfrak{f}$, 
$\mathfrak{g}$ and $\mathfrak{h}$ are $C^1$.
Denote by $X$ the set of $T$-periodic triples of \eqref{main} and by
$S$ the set of starting triples of the same equation, as above. Define
the map $\Pi:X\to S$ by
\[
\Pi(\lambda,x,y)=\big(\lambda,x|_{[-r,0]},y|_{[-r,0]}\big)
\]
and observe that $\Pi$ is continuous, onto and, since $\mathfrak{f}$,
$\mathfrak{g}$ and $\mathfrak{h}$ are smooth, it is also one to
one. Furthermore, by the continuous dependence on data, $\Pi^{-1}:S\to
X$ is continuous as well. Take
\[
S_\Omega=\Big\{ (\lambda ,\varphi,\psi)\in S:\mbox{the solution of
    \eqref{due.uno} is contained in $\Omega$} \Big\} ,
\]
so that $X\cap\Omega$ and $S_\Omega$ correspond under the
homeomorphism $\Pi:X\to S$. Thus, $S_\Omega$ is an open subset of $S$
and, consequently, we can find an open subset $W$ of $[0,\infty)\X\t
M\X\t N$ such that $S\cap W=S_\Omega$. This implies, as in \cite[Th.\
5.1]{FS09}, that
\[
\big\{(p,q)\in \giu{W^0}:g(p,q)=0\big\}
     =\big\{(p,q)\in\Omega_{\MxN}:g(p,q)=0\big\}.
\]
The excision property of the degree of tangent vector fields yields
\[
\deg\big(g,\giu{W^0}\big)=\deg\big(g,\Omega_{\MxN})\big)\neq 0.
\]
By Step 1 we deduce the existence of a connected set
\[
\Sigma\subseteq (S\cap W)\setminus\big\{(0,\su p,\su q)\in
W:g(p,q)=0\big\}
\]
whose closure in $S\cap W$ meets $\big\{(0,\su p,\su q)\in
W:g(p,q)=0\big\}$ and is not compact. Clearly, $\Gamma=\Pi^{-1}(\Sigma)$
satisfies the assertion.
\smallskip

\noindent\textbf{Step 3.} We now only need to remove the $C^1$-regularity assumption
on the maps $a$, $\mathfrak{f}$, $\mathfrak{g}$ and $\mathfrak{h}$ replacing it with
continuity. This is done by an approximation procedure that follows closely the one 
used in \cite[Th.\ 5.1]{FS09}. For this reason we skip the details.
\end{proof}

\begin{remark}\label{genRit}
  One can easily check that Theorem \ref{tunoRit} implies both
  \cite[Lemma 4.5]{BeCaFuPe} and \cite[Lemma 3.5]{BeCaFuPe2}, albeit
  in the less general case of boundaryless manifolds, which are valid
  for a single differential equation of the form
  \begin{equation*}
    \dot x(t)=\lambda \mathfrak{f}\big(t,x(t),x(t-r)\big),
  \end{equation*}
  where $\mathfrak{f}\colon\R\X M\X M\to\R^k$ is tangent to $M$. At
  the same time, Theorem \ref{tunoRit} extends \cite[Thm.\ 5.1]{FS09}
  (see also Th.\ 4.1 in \cite{BiSp13}) that applies to the
  differential equations of the following type:
  \begin{equation*}
    \dot y(t)=a(t)\mathfrak{g}(y)+\lambda \mathfrak{h}\big(t,y(t),y(t-r)\big).
  \end{equation*}
  where $\mathfrak{g}\colon M\to\R^k$ and $\mathfrak{h}\colon\R\X
  M\X M\to\R^k$ are tangent to $M$.
\end{remark}

\section{An application: A sunflower-like equation} \label{sec:sunflower-like} In
recent years there has been a growing interest around the dynamical
behavior of delay differential equations and in their possible use in
modeling biological and ecological systems. In particular, a certain
amount of research has been dedicated to the sunflower equation, i.e.
\begin{equation}\label{eq:sunflower}
  \ddot y(t) = -\frac{\alpha }{r} \dot y (t) - \frac{ \beta}{r} \sin \big(y(t-r)\big),
\end{equation}
where $\alpha$ and $\beta$ are experimental parameters and $r>0$ is
the finite time delay (see \cite{Girasole}). Somolinos in \cite{Somolinos} showed the
existence of periodic solutions to \eqref{eq:sunflower} for a certain
range of values for the involved parameters $\alpha$, $\beta$ and
$r$. This existence result covers both the cases of small and
large amplitude limit cycles generated by Hopf bifurcation. More recently,
Liu and Kalm\'ar-Nagy \cite{Liu} computed limit cycle amplitudes
and frequencies for \eqref{eq:sunflower}.
Other meaningful results related to this equation can be found in
\cite{DelayDiffEq, Casal-Freed, McD}. 

Here, we are concerned with a parametrized differential equation which arise quite 
naturally from \eqref{eq:sunflower}: It is obtained by assuming that the coefficient 
$-\alpha/r$ of $\dot y(t)$ is actually a real valued function $a\colon\R\to\R$ and 
setting $\lambda =-\beta/r$ as the coefficient of the second term in the left-hand 
side of \eqref{eq:sunflower}.
The parametrized equation under consideration reads as follows:
\begin{equation} \label{eq:sunflower-like} \ddot y(t) = a(t)\dot{y}
  (t) + \lambda \phi\big(y(t), y(t-r)\big),\,\, \lambda \geq 0,
\end{equation}
where $a$ and $\phi$ are continuous, $a$ is $T$-periodic with 
average $\avrg{a}\ne 0$. We point out that the assumption $\avrg{a}\ne 0$ 
serves to generalize the constant coefficient of $\dot y (t)$ in \eqref{eq:sunflower}. 
Ideally, $a(t)$ can be can be thought as a perturbation of the constant term $-\alpha/r$, 
namely $ a(t) = - \alpha/r + \epsilon (t) $ where $ \epsilon (t)$ is continuous,  
$T$-periodic and sufficiently small so that $\avrg{a}= -\alpha/r + \avrg{\epsilon}\ne 0$.

We wish to look at how the results of the previous section apply to
\eqref{eq:sunflower-like}, expecially in the case in which 
$\phi\big(y(t),y(t-r)\big)=\sin \big( y(t-r)\big)$.\smallskip

Let us now recall the notion of $T$-periodic pair (or $T$-pair for
brevity) for Equation~\eqref{eq:sunflower-like} and some related
facts.

A pair $(\lambda,y)\in [0,\infty)\X C^1(\R)$ is called a \emph{$T$-pair} for 
\eqref{eq:sunflower-like} if $y$ is a $T$-periodic solution of \eqref{eq:sunflower-like}. 
A $T$-pair $(\lambda, y)$ is \emph{trivial} if $y$ is constant and $\lambda=0$.


In order to study Equation~\eqref{eq:sunflower-like}, we introduce a transformation that 
allows us to rewrite this model in an equivalent but easier to handle form. We need the 
following technical lemma whose proof is a standard ODE argument which we provide for 
the sake of completeness.

\begin{lemma}\label{lem-link}
Let $a\colon\R\to\R$ be as in \eqref{eq:sunflower-like} and such that $\avrg{a}\neq 0$. 
Then, there exists a unique $T$-periodic $C^1$ function $\sigma\colon\R\to\R$ for which
\begin{equation}\label{eq.transf}
 a(t)=\frac{\dot\sigma(t)}{\sigma(t)} - \sigma(t),\qquad\text{for all $t\in\R$}.
\end{equation}
Clearly, $\sigma$ has constant sign so that, in particular $\avrg \sigma \ne 0$.
%
\end{lemma}

As a direct consequence, we have that Equation~\eqref{eq:sunflower-like} can be rewritten as
\begin{equation}\label{app1-bis}
  \ddot y(t)= \left(\frac{\dot\sigma (t)}{\sigma(t)}- \sigma(t)\right)\dot y(t) +
   \lambda \phi\big(y(t), y(t-r)\big),\quad \lambda\geq 0,
\end{equation}
with $\sigma$ chosen as in Lemma \ref{lem-link}.

\begin{proof}[Proof of Lemma~\ref{lem-link}] 
It is easy to verify by inspection that, for any $c\in\R$,
\[
 \zeta(t):= e^{-\int_0^t a(s) ds}\left( c - \int_0^te^{\int_0^s a(\ell)d\ell}ds\right),
\]
is a solution of equation
\begin{equation}\label{eq.inz}
 \dot \zeta(t) = -\zeta(t) a(t) - 1,
\end{equation}
which corresponds to \eqref{eq.transf} under the transformation $\zeta (t)=1/\sigma(t)$.
Clearly, $\zeta(0)=c$. Taking
\[
 c=c_0:=\frac{\int_0^Te^{\int_0^s a(\ell)d\ell}ds}{e^{-T\avrg{a}}-1}e^{-T\avrg{a}},
\]
(recall that $\avrg{a}\neq 0$)
we get $\zeta(0)=\zeta(T)$. Since the right-hand-side of \eqref{eq.inz} is $T$-periodic we 
obtain that $\zeta$ is $T$-periodic as well. In fact, the above is the only choice of $c$ for 
which $\zeta(0)=\zeta(T)$; thus \eqref{eq.inz} has a unique $T$-periodic solution.

 We need to prove that the function $t\mapsto 1/\zeta(t)$
is a $T$-periodic solution of \eqref{eq.transf}. It is sufficient to show that $z(t)\neq 0$
for all $t\in\R$. We consider the two possibilities $\avrg{a}>0$ and $\avrg{a}<0$ separately:

\noindent{\bfseries Case $\mathbf{\avrg{a}>0}$.}  Clearly, $e^{-T\avrg{a}}<1$ so that, since
$e^{-T\avrg{a}}\int_0^te^{\int_0^s a(\ell)d\ell}ds>0$ we have $c<0$. Now, being 
$e^{-\int_0^t a(s) ds}>0$ for all $t\in\R$, we get $\zeta(t)<0$ for all $t$.

\noindent{\bfseries Case $\mathbf{\avrg{a}<0}$.} In this case one has $e^{-T\avrg{a}}>1$, thus
\[  
  1<\frac{e^{-T\avrg{a}}}{e^{-T\avrg{a}}-1}.
\]
Since $t\mapsto e^{-\int_0^t a(s)ds}$ is a positive function we have:
\[
 \int_0^te^{\int_0^s a(\ell)d\ell}ds<\int_0^Te^{\int_0^s a(\ell)d\ell}ds
                 <\frac{e^{-T\avrg{a}}}{e^{T\avrg{a}}-1}\int_0^Te^{\int_0^s a(\ell)d\ell}ds,
\]
so that $\zeta(t)>0$.

Thus, in both cases, we find a $T$-periodic solution of \eqref{eq.transf}. The uniqueness, 
follows from the fact that if $t\mapsto\sigma(t)$ is a $T$-periodic solution of 
\eqref{eq.transf}, hence defined for all $t\in\R$, then $\sigma(t)\neq 0$. Then,
$t\mapsto 1/\sigma(t)$ is a $T$-periodic solution of \eqref{eq.inz} which, as discussed
above, is unique.
\end{proof}

\begin{remark}\label{sign-avg}
 From the proof of Lemma \ref{lem-link} it follows that $\sigma$ has (constantly) the opposite 
 sign of the average  $\avrg{a}$. This fact has the obvious consequence that the signs of the
 averages of $\sigma$ and of $1/\sigma$ coincide with $-\sign (\avrg{a})$. 
 For the average $\avrg{\sigma}$ of $\sigma$ something more precise can be be deduced by the 
 following simple argument:
\[
 \avrg{a}=\frac{1}{T}\int_0^Ta(t)\dif t
         = \frac{1}{T}\int_0^T\frac{\dot \sigma(t)}{\sigma(t)}\dif t 
                                 -  \frac{1}{T}\int_0^T\sigma(t)\dif t
         =-\avrg{\sigma}.
\]
In fact, 
\[
\int_0^T\frac{\dot\sigma(t)}{\sigma(t)}\dif t 
          = \ln\big(\sigma(T)\big)-\ln\big(\sigma(0)\big)=0
\]
because of the $T$-periodicity of $\sigma$.
\end{remark}

\smallskip
To investigate Equation \eqref{eq:sunflower-like} or, equivalently, \eqref{app1-bis} we follow 
the approach used in \cite[\S 5]{Sp06}. Along this path we find convenient to treat a more general 
class of equations, i.e.\
\begin{equation}\label{app1}
  \ddot y(t)= \left(\frac{\dot\gamma (t)}{\gamma(t)} - 
    \gamma(t) g\big(y(t)\big)\right)\dot y(t) +
  \lambda f\big(t, y(t), y(t-r)\big),\quad \lambda\geq 0,
\end{equation}
where $f\colon\R\X\R^2\to\R$ is continuous and $T$-periodic in $t$, $\gamma\colon\R\to\R$ is 
$T$-periodic and nonzero, and $g\colon\R\to\R$ is $C^1$. 

Introducing a new variable $x$, Equation~\eqref{app1}
can be equivalently rewritten in $\R^2$ (as in the so-called Li\'enard
plane technique) as follows:
\begin{equation}\label{app2}
  \left\{
    \begin{array}{l}
      \dot x(t)=\lambda f\big(t, y(t), y(t-r)\big)\gamma^{-1}(t),\\
      \dot y(t)=\big( x - G(y)\big) \gamma(t),
    \end{array}
  \right.\qquad
  \lambda\geq 0,
\end{equation}
where $G(y)$ is a primitive of $g(y)$ and $\gamma$ is as in Lemma \ref{lem-link}. 
Indeed, taking the derivative of the second equation in \eqref{app2}, we have
\begin{equation*} 
\begin{aligned}
\ddot y(t) =&\, \dot \gamma(t) \Big( x(t)- G\big(y(t)\big) \Big) + 
\gamma(t) \dot x(t) - g\big(y(t)\big)\dot y(t)\\
= & \, \frac{\dot\gamma(t)}{\gamma(t)} \dot y(t) -g\big(y(t)\big)\dot y(t)
+ \lambda f\big(t, y(t), y(t-r)\big).
\end{aligned}
\end{equation*}
By this relation, one can easily see that \eqref{app2} is equivalent to \eqref{app1}.
Because of this equivalence, Theorem \ref{tunoRit} can be applied to \eqref{app1}:

\begin{proposition}\label{tapp}
Let $f$, $g$ and $\gamma$ be as in \eqref{app1}, and let $\Omega\subseteq[0,\infty)\X C^1_T(\R)$
be open. Define the open subset of $C_T(\R\X\R)$
\[
\h\Omega:=\big\{(\lambda,\varphi,\psi)\in[0,\infty)\X C_T(\R\X\R):
 (\lambda,\varphi)\in\Omega\big\},
\]
and, according to the notation of Theorem \ref{tunoRit},
\[
 \h\Omega_{\R^2}=\big\{(\lambda,p,q)\in[0,\infty)\X\R\X\R:(\lambda,\h p,\h q)\in\h\Omega\big\}.
\]
Consider the vector field $\nu$ in $\R^2$, given by
\[
\nu(p,q):=\big(\bar w(q), \, p-G(q)\big),
\]
with $\bar w(q):=\frac{1}{T}\int_0^Tf(t,q,q)\gamma^{-1}(t)\dif{t}$. Assume that $\nu$ is 
admissible in $\Omega_{\R^2}$ for the degree and that $\deg(\nu,\Omega_{\R^2})\neq 0$.
Then, there exists a connected set of nontrivial $T$-pairs for \eqref{app1} whose closure meets 
the set 
$
\big\{(0,\h p)\in\Omega: \bar w(p)=0\big\}
$
and, is not compact. 
\end{proposition}

\begin{proof}
By Theorem \ref{tunoRit}, there exists a connected set $\Gamma$ of nontrivial $T$-triples for
\eqref{app2} whose closure meets the set
\[
\left\{(0,\h p,\h q)\in\h{\Omega}: \bar w(q)=0,\;  p=G(q)\right\}
\]
and is not compact.

Observe that to any $(\lambda,y,z)\in\Gamma$ one can associate the nontrivial $T$-pair 
$(\lambda,y)$ for \eqref{app1}. In this way, one gets a connected set of nontrivial 
$T$-pairs for \eqref{app1} whose closure meets the set 
$
\left\{(0,\h p)\in\Omega:\bar w(p)=0\right\}
$
and is not compact. 
\end{proof}

\begin{example}
Consider Equation \eqref{app1} with $\gamma(t)=\sin(t)+ 2$ and $g(y)\equiv1$; that is: 
\begin{equation}\label{ex1}
 \ddot x(t)=\left(\frac{\cos(t)}{\sin(t)+ 2}-(\sin(t)+ 2)\right)\dot x(t)+\lambda x(t-r).
\end{equation}
Take $T=2\pi$. Clearly, the average $\avrg\gamma=2$ and, for any $q\in\R$,
\[ 
\bar w(q)=\frac{1}{2\pi}\int_0^{2\pi}\frac{q}{\sin (t)+2}dt=\frac{q}{\sqrt{3}}.
\]
Let $\Omega=[0,\infty)\X C_T^1(\R)$.  The vector field $\nu(p,q)=\big(q/\sqrt{3},p - q\big)$
is clearly admissible in $\h\Omega_{\R^2}=\R^2$ and has degree $1$. Then, by Proposition 
\ref{tapp}, there exists a connected set of nontrivial $2\pi$-pairs for \eqref{ex1} whose 
closure meets the set
\[
\big\{(0,\h p)\in[0,\infty)\X C^1_T(\R):\bar w(p)=0\big\}
\]
and is not compact.
\end{example}

\begin{remark}
When $\gamma(t)\equiv 1$,  the system of equations  \eqref{app2} reduces to
\begin{equation*}
  \left\{
    \begin{array}{l}
      \dot x(t)=\lambda f\big(t, y(t), y(t-r)\big),\\
      \dot y(t)=\big( x - G(y)\big),
    \end{array}
  \right.\qquad
  \lambda\geq 0,
\end{equation*}
which is equivalent to the equation
\begin{equation} \label{eq:lienard-sunflower}
 \ddot y(t) = -g\big(y(t)\big)\dot y(t) + \lambda f\big(t, y(t), y(t-r)\big).
\end{equation}
in the particular case when $f\big(t,y(t),y(t-r)\big)=f\big(y(t-r)\big)$, Equation 
\eqref{eq:lienard-sunflower} gives a so-called delayed Li\'enard equation (or Li\'enard 
sunflower-type equation) see, e.g., \cite{AcoLiz, DelayDiffEq, OmZan,
    Zhang, Zhang2, WeiHuang, XuLu}.
Clearly, Proposition \ref{tapp} applies (for the non-delayed case, see \cite{Sp06}).  
\end{remark}

When $f$ does not depend on $t$, Proposition \ref{tapp} combined with Lemma \ref{lem-link}
implies the main result of this section concerning Equation \eqref{eq:sunflower-like}:

\begin{theorem}\label{capp}
Let $\phi$ and $a$ be as in \eqref{eq:sunflower-like} and let $\Omega\subseteq[0,\infty)\X C^1_T(\R)$
be open. Take $W_\Omega:=\big\{p\in\R:(0,\h p)\in\Omega\big\}$, and let $w(q):=\phi(q,q)$. Assume 
that $\deg(w,W_\Omega)$ is well-defined and nonzero. Then,  there exists a connected set of nontrivial 
$T$-pairs for \eqref{eq:sunflower-like}, whose closure meets the set $\{(0,\h p)\in\Omega:w(p)=0\}$ and 
is not compact. 
\end{theorem}

\begin{proof}
By Lemma \ref{lem-link} there exists a unique $T$-periodic function of constant sign 
$\sigma\colon\R\to\R$ such that $a(t)=\dot\sigma(t)/\sigma(t)-\sigma(t)$. Therefore, 
\eqref{eq:sunflower-like} can be written in the form \eqref{app1-bis} with $f(t,p,q)=\phi(p,q)$
for all $(p,q)\in\R^2$.

Take $G(y)=y$. Then the maps $\bar w$ and $\nu$ of Proposition \ref{tapp} become, respectively
\[
 \bar w(q)=\frac{1}{T}\int_0^T\frac{\phi(q,q)}{\sigma(t)}\dif{t}
 =\phi(q,q)\,\frac{1}{T}\int_0^T\frac{\dif{t}}{\sigma(t)}
 \quad\text{and}\quad \nu(p,q) = \big(\bar w(q),p-q\big).
\]
Let $\h\Omega_{\R^2}$ as in Proposition \ref{tapp}, with $\gamma=\sigma$. One easily checks that
\begin{equation*}
 \deg\big(\nu,\h\Omega_{\R^2}\big)
    =-\sign\left(\frac{1}{T}\int_0^T\frac{\dif{t}}{\sigma(t)}\right)\deg(w,W_\Omega)
                   =\sign(\avrg a)\deg(w,W_\Omega),
\end{equation*}
the last equality being a consequence of Remark \ref{sign-avg}. Thus, $\deg(w,W_\Omega)\neq 0$ implies
$\deg\big(\nu,\h\Omega_{\R^2}\big)\neq 0$. The assertion now follows from Proposition \ref{tapp}.
\end{proof}

In the following example we consider the case of Equation~\eqref{eq:sunflower-like}
when the perturbing term $\phi\big(y(t),y(t-r)\big)=\sin\big( y(t-r)\big)$.

\begin{example}[Sunflower-like equation]
  Consider the following scalar equation:
  \begin{equation}\label{ex2}
    \ddot x(t) = a(t)\dot x(t)+\lambda\sin\big(x(t-r)\big).
  \end{equation}
  where $a\colon\R\to\R$ is as in \eqref{eq:sunflower-like}. Let $\Omega$ be the open
  subset of $[0,\infty)\X C_T^1(\R)$ given by $\Omega=[0,\infty)\X C_T^1\big((-1,1)\big)$,
  and let $W_\Omega$ be as in Theorem \ref{capp}. Let $T=2\pi$. One immediately checks 
  that $\deg(w,W_\Omega)=1$, where 
\[
  w(p)=\frac{1}{2\pi}\int_0^{2\pi}\sin(p)\dif t =\sin(p).
\]
  By Theorem \ref{capp} there exists a connected set of nontrivial $T$-pairs for 
\eqref{eq:sunflower-like}, whose closure meets the set $\{(0,\h p)\in\Omega:w(p)=0\}$ and is 
not compact. 
\end{example}

\end{document}